\pdfoutput=1
\RequirePackage{ifpdf}
\ifpdf 
\documentclass[pdftex]{sigma}
\else
\documentclass{sigma}
\fi

\usepackage{tikz}
\usetikzlibrary{shapes,matrix,arrows,calc}

\tikzstyle cross=[preaction={draw=white, -, line width=4pt}, thick]
\tikzstyle vcross=[preaction={draw=blue!30, -, line width=4pt}, thick]
\tikzstyle normal=[thick]
\tikzstyle chord=[densely dotted, thick]
\tikzstyle zero=[ultra thick, gray]
\tikzstyle zerocross=[preaction={draw=white, -, line width=4pt}, ultra thick, gray]
\tikzstyle point=[draw,circle,inner sep=1,fill=black]
\tikzstyle map=[draw,rectangle,fill=white]
\tikzstyle petitpoint=[draw,circle,inner sep=0.3,fill=black]

\newcommand{\negative}[3][-]{\draw[normal,#1] (#2+1,-#3).. controls (#2+1,-#3-0.3) and (#2,-#3-0.7)..(#2,-#3-1); \draw[cross,#1] (#2,-#3).. controls (#2,-#3-0.3) and (#2+1,-#3-0.7)..(#2+1,-#3-1);}
\newcommand{\positive}[3][-]{ \draw[normal,#1] (#2,-#3).. controls (#2,-#3-0.3) and (#2+1,-#3-0.7)..(#2+1,-#3-1);\draw[cross,#1] (#2+1,-#3).. controls (#2+1,-#3-0.3) and (#2,-#3-0.7)..(#2,-#3-1);}

\newcommand{\up}[3][-]{
\draw[normal,#1] (#2,-#3).. controls (#2,-#3-0.3) and (#2+0.2,-#3-0.5).. (#2+0.5,-#3-0.5);
\draw[normal] (#2+0.5,-#3-0.5).. controls (#2+0.8,-#3-0.5) and (#2+1,-#3-0.3).. (#2+1,-#3);}
\newcommand{\ap}[3][-]{
\draw[normal,#1] (#2,-#3-1).. controls (#2,-#3-0.7) and (#2+0.2,-#3-0.5).. (#2+0.5,-#3-0.5);
\draw[normal] (#2+0.5,-#3-0.5).. controls (#2+0.8,-#3-0.5) and (#2+1,-#3-0.7).. (#2+1,-#3-1);}

\newcommand{\straight}[3][-]{\draw[normal,#1] (#2,-#3) -- (#2,-#3-1);}

\newcommand{\triup}[3]{
	\draw[normal] (#1,-#2)-- (#1+0.5,-#2-0.5);
	\draw[normal] (#1+1,-#2)-- (#1+0.5,-#2-0.5);
	\draw[normal] (#1+0.5,-#2-0.5)-- (#1+0.5,-#2-1);
	\node[map] at (#1+0.5,-#2-0.5) {#3};
}
\newcommand{\tridown}[3]{
	\draw[normal] (#1+0.5,-#2)-- (#1+0.5,-#2-0.5);
	\draw[normal] (#1+1,-#2-1)-- (#1+0.5,-#2-0.5);
	\draw[normal] (#1,-#2-1)-- (#1+0.5,-#2-0.5);
	\node[map] at (#1+0.5,-#2-0.5) {#3};
}
\newcommand{\comult}[2]{
	\draw[normal] (#1+0.5,-#2)--(#1+0.5,-#2-0.5);
	\ap{#1}{#2}
}

\newcommand{\mult}[2]{
	\up{#1}{#2}
	\draw[normal] (#1+0.5,-#2-0.5)--(#1+0.5,-#2-1);
}

\newcommand{\tik}[1]{\begin{tikzpicture}[baseline=(current bounding box.center)] #1 \end{tikzpicture} }

\numberwithin{equation}{section}
\newcommand{\ind}{\operatorname{Ind}}
\newcommand{\ca}{\circlearrowleft}
\newcommand{\mf}[1]{\mathfrak{#1}}
\newcommand{\id}{\operatorname{id}}
\newcommand{\h}{\hslash}
\newcommand{\ot}{\otimes}
\newcommand{\cO}{\mathcal O}
\newcommand{\cA}{\mathcal {A}}
\newcommand{\cS}{\mathcal {S}}
\newcommand{\kk}{\mathbb K}
\renewcommand{\hom}{\operatorname{Hom}}
\newcommand{\vect}{\operatorname{vect}}

\newcommand{\coev}{\operatorname{coev}}
\newcommand{\ev}{\operatorname{ev}}
\newtheorem{thm}{Theorem}[section]
\newtheorem{cor}[thm]{Corollary}
\newtheorem{prop}[thm]{Proposition}
\newtheorem{lem}[thm]{Lemma}
\theoremstyle{definition}
\newtheorem{rmk}[thm]{Remark}

\begin{document}


\newcommand{\arXivNumber}{1604.08450}

\renewcommand{\PaperNumber}{088}

\FirstPageHeading

\ShortArticleName{A Duf\/lo Star Product for Poisson Groups}

\ArticleName{A Duf\/lo Star Product for Poisson Groups}

\Author{Adrien BROCHIER}

\AuthorNameForHeading{A.~Brochier}

\Address{MPIM Bonn, Germany}
\Email{\href{mailto:abrochier@mpim-bonn.mpg.de}{abrochier@mpim-bonn.mpg.de}}
\URLaddress{\url{http://abrochier.org}}

\ArticleDates{Received May 18, 2016, in f\/inal form September 05, 2016; Published online September 08, 2016}

\Abstract{Let $G$ be a f\/inite-dimensional Poisson algebraic, Lie or formal group. We show that the center of the quantization of $G$ provided by an Etingof--Kazhdan functor is isomorphic as an algebra to the Poisson center of the algebra of functions on $G$. This recovers and generalizes Duf\/lo's theorem which gives an isomorphism between the center of the enveloping algebra of a f\/inite-dimensional Lie algebra $\mf a$ and the subalgebra of ad-invariant in the symmetric algebra of $\mf a$. As our proof relies on Etingof--Kazhdan construction it ultimately depends on the existence of Drinfeld associators, but otherwise it is a fairly simple application of graphical calculus. This shed some lights on Alekseev--Torossian proof of the Kashiwara--Vergne conjecture, and on the relation observed by Bar-Natan--Le--Thurston between the Duf\/lo isomorphism and the Kontsevich integral of the unknot.}


\Classification{20G42; 17B37; 53D55}

\section{Introduction}

Let $\kk$ be a f\/ield of characteristic~0 and $G$ be a Lie, algebraic or formal group over $\kk$. A~multiplicative Poisson structure on $G$ is a Poisson structure such that the multiplication map $G\times G\rightarrow G$ is a Poisson map. This leads to the notion of Poisson Lie, Poisson algebraic or Poisson formal group depending on the context. If $\cO(G)$ is the Hopf algebra of $C^{\infty}$, regular or formal functions on $G$ then a multiplicative Poisson structure on $G$ turns $\cO(G)$ into a Poisson Hopf algebra.

A quantization of $G$ is a Hopf algebra which is a quantization of $\cO(G)$ as a Poisson algebra and whose coproduct reduces to the one of $\cO(G)$ at $\h=0$.

In~\cite{Etingof1996} Etingof--Kazhdan associate to any Drinfeld associator a functorial way to quantize Lie bialgebras and Poisson formal groups. Their construction can be applied to Poisson Lie and algebraic groups as well. Our main result is the following (Theorem~\ref{thm:main} below):
\begin{thm}\label{thm:intro}
	Let $\cO_\h(G)$ be the quantization of a finite-dimensional Poisson group $G$ obtained from an Etingof--Kazhdan functor. Then the center of $\cO_\h(G)$ is isomorphic as an algebra to the trivial $\kk[[\h]]$ extension of the Poisson center of $\cO(G)$.
\end{thm}

If $\mf a$ is a f\/inite-dimensional Lie algebra, then $G=\mf a^*$ as an abelian group is a Poisson algebraic group, the Poisson structure on $\cO(G)=S(\mf a)$ being induced by the Lie bracket of $\mf a$. The Poisson center is identif\/ied with the sub-algebra of invariant $S(\mf a)^{\mf a}$. On the other hand one can show that the quantization of $G=\mf a^*$ can be specialized at $\h=1$, and becomes isomorphic to the enveloping algebra of $\mf a$. Hence we get Duf\/lo's theorem~\cite{Duflo1977} as a corollary:
\begin{cor}
Let $\mf a$ be a finite-dimensional Lie algebra, then there is an isomorphism of algebras
\begin{gather*}
	S(\mf a)^{\mf a}\cong U(\mf a)^{\mf a}.
\end{gather*}
\end{cor}

In fact there are two constructions of a quantization of $G$ in~\cite{Etingof1996}: the f\/irst one comes from a~natural construction of a certain f\/iber functor out of a coalgebra in a braided monoidal category constructed from a Drinfeld associator. This construction is not, however, functorial, and the authors then modify it in an appropriate way to make it so.

Our f\/irst result (Theorem~\ref{thm:Poisson}) is a fairly simple graphical proof that the restriction of the star product of the f\/irst construction to the Poisson center of $\cO(G)$ is undeformed, i.e., the same as the original product. This result alone guarantee the existence of a quantization of~$G$ fulf\/illing the conclusion of Theorem~\ref{thm:intro}.

The second step is to show that the two constructions yields isomorphic algebras (and in fact, isomorphic Hopf algebras). This isomorphism is higly non-trivial and depends on the underlying associator. Its existence relies on the fact that the two f\/iber functors at hand are isomorphic as monoidal functors. However, the proof of this fact given in~\cite{Etingof1996} appears to be incorrect: the isomorphism between those functors given there is not monoidal due to the contribution of certain non-trivial associativity constraints. Taking the associativity constraints into account and correcting this construction requires a detailed discussion of dualities in categories constructed from Drinfeld associators and a generalization of a result by Le--Murakami on the behaviour of the Kontsevich integral under taking parallels of tangles (Proposition~\ref{prop:dual}). In particular, one has to modify the ordinary coevaluation map in the Drinfeld category using a very specif\/ic element closely related to the Kontsevich integral of the unknot, which explains its relation with the diagrammatic Duf\/lo isomorphism of~\cite{MR1988280}.

For the sake of concreteness we work with Lie bialgebras in the category of vector spaces, but our proof also applies in the following situation: if $\mf g$ is a Lie bialgebra in a linear symmetric monoidal category $\cS$ which is dualizable as an object of $\cS$, then one can def\/ine the algebra of functions over the Poisson formal group of $\mf g$ as the object $\widehat{S}(\mf g^*)$ equipped with its standard multiplication and the coproduct induced by the Baker--Campbell--Hausdorf\/f formula. Then Theorem~\ref{thm:intro} is also valid in this setting. In particular we get a version of the Duf\/lo Theorem for dualizable Lie algebras in arbitrary symmetric monoidal categories.

We note that an analog of Theorem~\ref{thm:intro} holds for Kontsevich's deformation quantization of an arbitrary Poisson manifold $M$, which also implies Duf\/lo's theorem~\cite{Cattaneo2002,Kontsevich2003,Manchon2003}. While it is true for duals of f\/inite-dimensional Lie algebras, to the best of the author's knowledge it is not known whether Kontsevich's star product on a Poisson group seen as a Poisson manifold is isomorphic to the one coming from Etingof--Kazhdan construction (if for the former one uses Tamarkin's construction~\cite{Tamarkin1998, Tamarkin1999}, and if one chooses the same associator in both cases). An af\/f\/irmative answer to this question would thus give another proof of our result. Another motivation for this paper is Alekseev--Torrossian proof of the Kashiwara--Vergne conjecture~\cite{Alekseev2012}. Roughly speaking they show that solutions of this conjecture are essentially the same as universal twist-quantization \emph{\`a la} Drinfeld of duals of f\/inite-dimensional Lie algebras. The Kashiwara--Vergne conjecture implies (but is much stronger than) Duf\/lo's theorem. This raises the question of whether there is a direct proof of Duf\/lo's theorem in Etingof--Kazhdan formalism and whether it can be generalized to other Poisson groups. This paper is thus an af\/f\/irmative answer to this question.

\section{Poisson groups}
To any f\/inite-dimensional Lie algebra $\mf g$ over $\kk$ is associated a formal group, whose algebra of functions is by def\/inition the dual $U(\mf g)^*$ of the enveloping algebra. This is a topological Hopf algebra, which as an algebra is identif\/ied with the degree completion $\widehat{S}(\mf g^*)$ of the symmetric algebra of $\mf g^*$, and whose coproduct
\begin{gather*}
\Delta_0\colon \ \widehat{S}(\mf g^*)\longrightarrow \widehat{S}(\mf g^*)\hat\ot \widehat{S}(\mf g^*):=\widehat{S}(\mf g^*\times \mf g^*)
\end{gather*}
is given by
\begin{gather*}
	\Delta_0(f)(x,y):=f(\operatorname{BCH}(x,y)),
\end{gather*}
where $\operatorname{BCH}$ is the Baker--Campbell--Hausdorf\/f series.

If $G$ is a Lie or algebraic group with Lie algebra $\mf g$ then the algebra $\cO(G)$ of $C^\infty$ or regular functions on $G$ is a Hopf algebra with coproduct\footnote{In the Lie case, the tensor product is the completed one: $\cO(G)\hat\ot \cO(G):=\cO(G\times G)$.}
\begin{gather*}
	\Delta_0(f)(X,Y):=f(XY).
\end{gather*}
Let $I$ be the augmentation ideal of $\cO(G)$, i.e., the ideal of functions vanishing at the identity. Then the Hopf algebra of functions on the formal group associated with $\mf g$ is isomorphic to the $I$-adic completion of $\cO(G)$.

A multiplicative Poisson structure on a formal, Lie or algebraic group $G$ is a Poisson bracket on $\cO(G)$ for which the coproduct is a map of Poisson algebra. In that case we say that $\cO(G)$ is a Poisson-Hopf algebra.
A quantization of $G$ is an Hopf algebra $(\cO_\h(G),\star_\h,\Delta_\h)$ over the ring of formal power series $\kk[[\h]]$ such that:
\begin{itemize}\itemsep=0pt
	\item $\cO_\h(G)\cong \cO(G)[[\h]]$ as a $\kk[[\h]]$-module,
	\item $\forall\, f, g \in \cO_\h(G)$, $f\star_\h g - g\star_\h f=\h\{f,g\}+O(\h^2)$,
	\item $\Delta_\h=\Delta_0+O(\h)$.
\end{itemize}

A Lie bialgebra structure on $\mf g$ is a linear map $\delta\colon \mf g\rightarrow \wedge^2 \mf g$ such that the dual map is a Lie bracket on $\mf g^*$ and
\begin{gather*}
	\delta([a,b])=[a\ot 1+1\ot a,\delta(b)]+[\delta(a),b\ot 1+1\ot b].
\end{gather*}
A multiplicative Poisson structure on $G$ induces a Lie bialgebra structure on $\mf g$, and conversely if $\mf g$ is a Lie bialgebra then the cobracket induces a multiplicative Poisson structure on $G$~\cite{Drinfeld1987}.

Let $\mf d$ be the double of $\mf g$. As a vector space this is $\mf g \oplus \mf g^*$ and its Lie bracket is determined by the following conditions:
\begin{itemize}\itemsep=0pt
	\item the inclusions $\mf g\rightarrow \mf d$ and $\mf g^* \rightarrow \mf d$ are Lie algebra maps,
	\item the canonical pairing on $\mf d$ is ad-invariant.
\end{itemize}

By the PBW theorem, there is a vector space isomorphism
\begin{gather*}
	U(\mf g)\cong U(\mf d)/U(\mf d)\mf g^*,
\end{gather*}
which turns $U(\mf g)$ into a $\mf d$-module (this coincides with the universal Verma module $M_-$ of the next section) hence $\mf d$ acts on the algebra of function on the formal group of $\mf g$. If $\mf g$ is the Lie algebra of a Lie or algebraic group $G$, then this action is by vector f\/ields on $G$, hence it can be globalised to an action of $\mf d$ on $\cO(G)$ is the Lie and algebraic setting as well (this is the so-called dressing action~\cite{Semenov-Tian-Shansky1985}). A key fact for the constructions in the next sections is that the Poisson bracket on~$\cO(G)$ can be expressed as follows: let~$x_i$ be a basis of $\mf g$ and $x^i$ be the dual basis of~$\mf g^*$, then
\begin{gather*}
	\forall\, f,g \in \cO(G),\qquad \{f,g\}=\sum_i \big(x^i\cdot f\big)(x_i\cdot g).
\end{gather*}
In particular, the following holds~\cite{Semenov-Tian-Shansky1985}:
\begin{prop}\label{prop:dressing}
The Poisson center of $\cO(G)$ coincides with the subalgebra $\cO(G)^{\mf g^*}$ of invariant under the dressing action of $\mf g^* \subset \mf d$.
\end{prop}

\section[Two quantizations of $G$]{Two quantizations of $\boldsymbol{G}$}
In this section we recall the main construction of~\cite{Etingof1996}, in a slightly dif\/ferent form inspired by~\cite{Severa2016}. Let $G$ be a Poisson group, $\mf g$ its Lie bialgebra and $\mf d$ the Drinfeld double of $\mf g$ and $t \in \mf d^{\ot 2}$ the canonical element.

\subsection{The Drinfeld category}

Let $\cA$ be the category whose objects are $\mf d$-modules and whose morphisms are def\/ined by
\begin{gather*}
	\hom_\cA(U,V):=\hom_{\mf d}(U,V)[[\h]].
\end{gather*}

Let $\Phi$ be a Drinfeld associator over $\kk$. Recall that this is a group-like element of the formal completion $\kk\langle\langle X,Y\rangle\rangle$ of the free associative algebra on two generators, satisfying the pentagon and the hexagon equation~\cite{Drinfeld1990a}. Let $\tilde\Phi$ be the image of $\Phi$ through the algebra morphism $\kk\langle\langle X,Y\rangle\rangle \rightarrow U(\mf d)^{\ot 3}[[\h]]$ induced by $X\mapsto \h t\ot 1$ and $Y\mapsto 1\ot \h t$.
	\begin{thm}[\cite{Drinfeld1990,Drinfeld1990a}]
For any $U,V,W \in \cA$ define a map
\begin{gather*}
	\alpha_{U,V,W}\colon \ (U\ot V)\ot W \rightarrow U\ot (V \ot W)
\end{gather*}
by the action of $\tilde \Phi$ and a map
\begin{gather*}
	\beta_{U,V}\colon \ U\ot V \rightarrow V \ot U
\end{gather*}
by $\beta_{U,V}=\exp(\h t/2)\circ P_{U,V}$ where $P_{U,V}(u\ot v)=v\ot u$. Then $\cA$ with its ordinary tensor product, associativity constraint~$\alpha$ and commutativity constraint $\beta$ is a braided tensor category.
\end{thm}

We will need a few results about duality in $\cA$. Note that $\cA$ is not a rigid category since we allow inf\/inite-dimensional modules, for which the coevaluation involves inf\/inite sums, and in particular not every map in $\cA$ has a well def\/ined transpose. However, it is easily checked that in all our computations involving duality, only f\/initely many terms of the coevaluation map will contribute. The existence of a suitable rigid structure on the sub-category of $\cA$ consisting of f\/inite-dimensional modules is discussed in~\cite{Cartier1993,Drinfeld1990,Kassel1998}, and in a dif\/ferent language but using a~dif\/ferent normalization that we will need in~\cite{Le1995,Le1997}.

Note that for a $\mf d$-module $V$, the ordinary evaluation and coevaluation are morphisms in $\cA$ but the presence of the non-trivial associativity constraints implies that they fail to satisfy the zig-zag identity required in order to def\/ine a duality in a monoidal category. Namely, write $\tilde \Phi=\sum_i x_i \ot y_i \ot z_i$ and set
\begin{gather*}
	\nu=\bigg(\sum_i x_iS(y_i)z_i\bigg)^{-1} \in U(\mf d)^{\mf d}[[\h]],
\end{gather*}
where $S$ is the antipode of $U(\mf d)$. Then the action of $\nu$ induces an automorphism of the identity functor in $\cA$ so that for any $V \in \cA$
\begin{gather*}
	(\id \ot \ev)\circ	\alpha_{V,V^*,V}\circ (\coev \ot \id)=\nu_V^{-1},
\end{gather*}
where $\ev$ (resp.~$\coev$) is the $\kk[[\h]]$-linear extension of the standard map $V^* \ot V \rightarrow \kk$ (resp.\ $\kk\rightarrow V \ot V^*$).
\begin{rmk}
	The element $\nu$ is essentially a specialization of the Kontsevich integral of the unknot, and can be shown to be independent of the choice of the associator $\Phi$~\cite{Le1996}.
\end{rmk}
For any $\alpha,\beta \in U(\mf d)^{\mf d}[[\h]]$ of the form $1+O(\h)$ and such that $\alpha \beta=\nu$ the maps
\begin{gather*}
	\tik{\up[<-]{0}{0}} :=	\ev\circ	(\id \ot \alpha_V),\qquad  \tik{\ap[<-]{0}{0}} :=\coev\circ (\beta_V \ot \id)
\end{gather*}
induces a duality in $\cA$. Note that for any choice of $\alpha$, $\beta$ the dual of $V$ is its dual as a $\mf d$-module and the dual of a map is also its ordinary dual. As objects in $\cA$ we have
\begin{gather*}
	(V\ot W)^* = W^* \ot V^*
\end{gather*}
but both sides are dual to $V\ot W$ in an \emph{a priori} dif\/ferent way. In other words, the chosen duality induces a canonical automorphism
\begin{gather*}
	(V\ot W)^* = W^* \ot V^*\longrightarrow W^* \ot V^*
\end{gather*}
given by
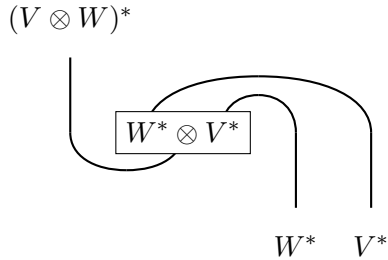
\begin{gather*}
	\tik{
		\straight{0}{0} \ap{2}{0}
		\begin{scope}[xscale=3,yscale=1.5]
			\ap{0.3333}{-0.333}
		\end{scope}
		\begin{scope}[xscale=1.5]
			\up{0}{1}
		\end{scope}
		\node[rectangle,draw,fill=white] at (1.5,-1) {$W^* \ot V^*$};
		\node at (0,0.5) {$(V\ot W)^*$};
		\node at (3,-2.5) {$W^* $};
		\node at (4,-2.5) {$V^*$};
		\straight{3}{1}\straight{4}{1}
		}
\end{gather*}
which fails to be the identity in general. This fact is closely related to the failure of the Kontsevich integral to be compatible with the operation of taking parallel for arbitrary tangles, which has been investigated by Le--Murakami~\cite{Le1997}, who show the following result: since $\nu=1+O(\h)$ it has a unique square root $\nu^{\frac12}$ of the same form. Recall that an associator $\Phi$ is called even if it satisf\/ies
\begin{gather*}
	\Phi(-X,-Y)=\Phi(X,Y).
\end{gather*}
\begin{thm}[Le--Murakami]\label{thm:LM}
	If one uses an even associator in the construction of $\cA$, and set $\alpha=\beta=\nu^{\frac12}$, then the canonical isomorphism
	\begin{gather*}
		(V \ot W)^*\cong W^* \ot V^*
	\end{gather*}
	is the identity.
\end{thm}

Using this result, we prove:
\begin{prop}\label{prop:dual}
For any choice of an associator, there exists $\alpha$, $\beta$ such that for the corresponding duality, the canonical isomorphism
\begin{gather*}
		(V \ot W)^*\cong W^* \ot V^*
\end{gather*}
is the identity.
\end{prop}
\begin{proof}
	Let $\Phi$, $\Phi'$ be two associators and $\cA$, $\cA'$ the braided monoidal categories constructed from them. By~\cite{Drinfeld1990a,Le1996} there exists a strong monoidal structure $J$ on the identity functor inducing a~braided monoidal equivalence
\begin{gather*}
\cA' \simeq \cA.
\end{gather*}
The morphism $J$ is given by the action of a symmetric, $\mf d$-invariant element in $U(\mf d)^{\ot 2}[[\h]]$, still denoted by~$J$, which is given by a~universal formula, i.e., can be chosen independently of $\mf d$.

In particular, if one chooses any even associator $\Phi'$ one can transport the duality from~$\cA'$ to~$\cA$ using this equivalence. Namely we def\/ine
\begin{gather*}
	\tik{\up[<-]{0}{0}} :=	\ev\circ J_{V^*,V} 	\circ \big(\id \ot \nu^{\frac12}_V\big),\qquad  \tik{\ap[<-]{0}{0}} :=\coev\circ J^{-1}_{V,V^*} \circ \big(\nu^{\frac12}_V \ot \id\big).
\end{gather*}
Write
\begin{gather*}
	J =\sum f_i \ot g_i,\qquad J^{-1} =\sum \bar f_i \ot \bar g_i,
\end{gather*}
then this duality corresponds to the choice
\begin{gather*}
	\alpha_F=\sum S(f_i)\nu^{\frac12}g_i
\end{gather*} and
\begin{gather*}
	\beta_F=\sum \bar f_i \nu^{\frac12} S(\bar g_i)
\end{gather*}
(see, e.g.,~\cite{Drinfeld1990}).

For this choice of duality, the isomorphism $W^* \ot V^*\rightarrow (V\ot W)^* $ is given by
\begin{gather}\label{eq:dualTwist}
	(J_{V,W})^*J_{W^*,V^*},
\end{gather}
where we used the fact that by Theorem~\ref{thm:LM} the analog morphism is the identity in $\cA'$. Now since any associator $\Phi$ is group-like, we have $(S\ot S\ot S)(\Phi)=\Phi^{-1}$. Therefore, the element $(S \ot S)(J^{-1})$ satisf\/ies the same twist equation as $J$, and since we choose $J$ to be the specialization of an universal twist it follows from~\cite{Drinfeld1990a} that there exists an invertible central element $u \in U(\mf d)[[\h]]$ such that
\begin{gather*}
	(S\ot S)(J)=\big(u^{-1}\ot u^{-1}\big)\Delta(u)J^{-1}.
\end{gather*}
Now by def\/inition $(J_{V,W})^*$ is given by the action of $(S\ot S)(J)$, and using the fact that $J$ is symmetric, the map~\eqref{eq:dualTwist} is given by the action of $(u^{-1}\ot u^{-1})\Delta(u)$. Finally, setting
\begin{gather*}
	\alpha =\alpha_F u,\qquad \beta =\beta_Fu^{-1}
\end{gather*}
makes this map trivial.
\end{proof}

\subsection{Two f\/iber functors}

Let $M_+=\ind_{\mf g}^{\mf d} 1$ and $M_-=\ind_{\mf g^*}^{\mf g} 1$ be the universal Verma modules. By the PBW theorem as vector spaces $M_+\cong U(\mf g^*)$ and $M_-\cong U(\mf g)$, which implies that $M_{\pm}$ are coalgebras. We denote by $1_\pm$ the image of $1$ through those isomorphisms. Def\/ine functors $F^{\circlearrowleft}$ and $F$ from $\cA$ to the category $\vect_\h$ of topologically free $\kk[[\h]]$ modules by
\begin{gather*}
	F^{\circlearrowleft}(V) =\hom_\cA(M_+ \ot M_-, V), \qquad 	F(V) =\hom_\cA(M_+, V\ot M_-^*).
\end{gather*}
\begin{prop}
	The functors $F^\ca$ and $F$ are naturally isomorphic to the forgetful functor sending a $\mf d$-module $V$ to $V[[\h]]$ via
	\begin{gather*}
		f\longmapsto f(1_+ \ot 1_-)
	\end{gather*}
	and
	\begin{gather*}
		f \longmapsto (\id\ot 1_-)f(1_+)
	\end{gather*}
	respectively.

\end{prop}
The following observation is key to the present construction:
\begin{prop}\label{prop:coalg}
Equipped with their original coproduct and counit, $M_\pm$ are cocommutative coalgebras in $\cA$.
\end{prop}

Recall that if $\cA$ is any tensor category, a coalgebra structure on an object $C$ turns the functor $\hom_\cA(C,-)$ into a monoidal functor. If $\cA$ is braided and if $C_1$, $C_2$ are two coalgebras in $\cA$, then the tensor product $C_1\ot C_2$ has a natural coalgebra structure (the braided tensor product) given by
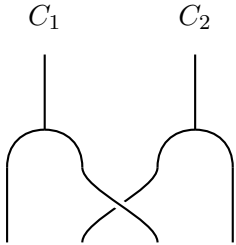
\begin{gather*}
	\tik{\straight{0.5}{0.5} \straight{2.5}{0.5}
	\ap{0}{1}\ap{2}{1}
	\straight{0}{2}\negative{1}{2}\straight{3}{2}
	\node at (0.5,0) {$C_1$};
	\node at (2.5,0) {$C_2$};
}
\end{gather*}
Hence, the functor $F^\ca$ has a natural monoidal structure which we denote by $J^\ca$. One checks that $J^\ca$ is actually invertible, i.e., the functor $F^\ca$ is strong monoidal. One can then formally dualize this structure to obtain a monoidal structure on $F$. Namely, let $J$ be the map
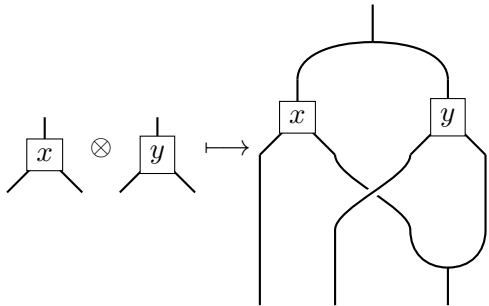
\begin{gather*}
	F(V)\ot F(W)\rightarrow F(V\ot W)
\end{gather*}
def\/ined by
\begin{gather*}
	\tik{\tridown{0}{0}{$x$}}\ot\tik{\tridown{0}{0}{$y$}}\longmapsto
	\tik{
		\begin{scope}[xscale=2]
		\comult{0.25}{0}
		\end{scope}
		\tridown{0}{1}{$x$}
		\tridown{2}{1}{$y$}
		\straight{0}{2}\positive{1}{2}\straight{3}{2}
		\mult{2}{3}\straight{0}{3}\straight{1}{3}
		}
\end{gather*}
Let $M_-^*$ be equipped with the multiplication dual (in $\vect$) to the comultiplication of $M_-$. We then have:
\begin{prop}\label{prop:algebra}
The object $M_-^*$ is an algebra in $\cA$, and $J$ induces a strong monoidal structure on $F$.
\end{prop}
We stress the fact that it is not straightforward that the coalgebra $M_-$ and the algebra $M_-^*$ are dual as objects in $\cA$. However this follows from Proposition~\ref{prop:dual}:
\begin{prop}\label{prop:transpose}
	If one chooses the duality in $\cA$ as in Proposition~{\rm \ref{prop:dual}}, then the multiplication of $M_-^*$ is dual in $\cA$ to the comultiplication of $M_-$.
\end{prop}
\begin{proof}
	Let $\Delta$ be the comultiplication of $M_-$. Then by def\/inition the multiplication it induces on the object $M_-^*$ in $\cA$ is given by the composition
	\begin{gather*}
		M_-^* \ot M_-^* \xrightarrow{\cong} (M_- \ot M_-)^* \xrightarrow{\Delta^*} M_-^*.
	\end{gather*}
	By Proposition~\ref{prop:dual} the f\/irst map is trivial, and since $M_-$ is cocommutative in $\cA$ the result follows.
\end{proof}
\begin{rmk}
	If one chooses a dif\/ferent duality, then it still follows from Proposition~\ref{prop:dual} that these two algebra structures are isomorphic.
\end{rmk}
Finally, we obtain
\begin{prop}\label{prop:monoidal}
The map given by
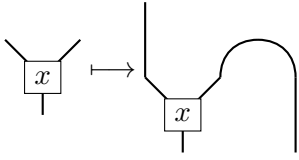
\begin{gather*}
	\tik{\triup{0}{0}{$x$}}\longmapsto \tik{\straight{0}{0}\ap{1}{0}\triup{0}{1}{$x$}\straight{2}{1}}
\end{gather*}
induces a natural monoidal isomorphism between $F^\ca$ and $F$.
\end{prop}
\begin{proof}
	As in~\cite[Proposition~9.7]{Etingof1996} this follows from the identity
	\begin{gather*}
		\tik{
	\ap{0.5}{0}
	\comult{0}{1}\straight{1.5}{1}
	}
	=
	\tik
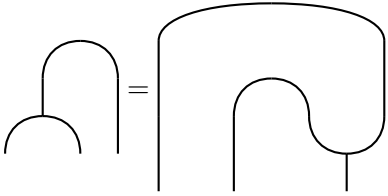
{
		\begin{scope}[xscale=3]
			\ap{0}{0}
		\end{scope}
		\straight{0}{1}\ap{1}{1}\straight{3}{1}
		\straight{0}{2}\straight{1}{2}\mult{2}{2}
	}
	\end{gather*}
	which holds thanks to Proposition~\ref{prop:transpose}. We note that this in~\cite{Etingof1996} this identity is claimed to be true with the ordinary coevaluation instead of the one coming from Proposition~\ref{prop:dual}, which is not correct.
\end{proof}

\subsection[Quantization of $G$]{Quantization of $\boldsymbol{G}$}

The same proof as for Proposition~\ref{prop:algebra} implies:
\begin{prop}
The algebra of function on $G$ equipped with the dressing action of $\mf d$ and its original algebra structure, is a commutative algebra in~$\cA$.
\end{prop}

Hence, let $\cO_\h(G)=F(\cO(G))$ and $\cO^\ca_\h(G)=F^\ca(\cO(G))$. Since both functors are monoidal those are algebras in $\vect_\h$, and since those functors are monoidally equivalent these algebras are isomorphic. The main result of~\cite{Etingof1996, Severa2016} is:
\begin{thm}\label{thm:EK}The following holds true:
	\begin{itemize}\itemsep=0pt
		\item The algebras $\cO_\h(G)$ and $\cO^\ca_\h(G)$ are Hopf algebras quantizing $G$. The coproduct on $\cO_\h(G)$ is given by\footnote{Once again for $G$ a Lie or formal group, those are topological Hopf algebras, i.e., the coproduct takes values in the completed tensor product.}
			\begin{gather*}
				F(\cO(G))\xrightarrow{\id\ot 1_{\cO(G)}} F(\cO(G)\ot \cO(G))\xrightarrow{\cong} F(\cO(G))\ot F(\cO(G)).
				\end{gather*}
		\item The map of Proposition~{\rm \ref{prop:monoidal}} induces an isomorphism of Hopf algebras $\cO_\h(G)\cong \cO_\h^\ca(G)$.
		\item The assignment $G \mapsto \cO_\h(G)$ induces a contravariant functor from the category of finite-dimensional Poisson $($Lie, algebraic or formal$)$ groups over $\kk$ to the category of Hopf algebras over~$\kk[[\h]]$.
	\end{itemize}
\end{thm}
\begin{rmk}
	We stress the fact that the construction of $\cO_\h(G)^\ca$ is not functorial.
\end{rmk}
\begin{rmk}
It is of course important that the algebra structure that we consider is part of a Hopf algebra structure, i.e., that we consider quantizations of $G$ as a Poisson group and not just as a Poisson manifold. Yet it is suf\/f\/icient for our purpose that the coproduct exists, which is why we refer the reader to~\cite{Etingof1996, Severa2016} for its construction.
\end{rmk}
\begin{rmk}\label{rmk:EK}
	To be precise, the construction of~\cite{Etingof1996} uses the functor
\begin{gather*}
		\hom_\cA(M_-,M_+^*\ot -)
	\end{gather*} applied to the coalgebra $M_-$. The authors then show that this def\/ines a functorial quantization of~the Lie bialgebra $\mf g$ and then def\/ine the quantization of the formal group of $\mf g$ as $\hom_\cA(M_-,M_+^*\ot M_-)^*$. The construction presented here is the correct way of dualizing the whole construction in order to obtain $\cO_\h(G)$ as the image of an algebra in~$\cA$, and coincides with the one given in~\cite{Severa2016}. Indeed, in \emph{loc.\ cit.} the author starts with the braided tensor product of two copies of~$\cO(G)$, and then applies the monoidal functor $V\mapsto V^{\mf g}$ to it. By Frobenius reciprocity this is
	\begin{gather*}
		\hom_\cA(M_+,\cO(G)\ot \cO(G)).
	\end{gather*}
	Replacing the second copy of $\cO(G)$ by its formal completion which, as an algebra in $\cA$ is by def\/inition $M_-^*$ one sees that the two constructions are the same.
\end{rmk}
\begin{rmk}
	Strictly speaking, what is claimed in~\cite{Etingof1996} is the existence of an Hopf algebra isomorphism between $F^\ca(M_-)$ and
	\begin{gather*}
		\hom_\cA(M_-,M_+^*\ot M_-).
	\end{gather*}
	Once again, this claim is valid if one uses the modif\/ied coevaluation of Proposition~\ref{prop:dual} instead of the ordinary one in the construction of an isomorphism between these functors. The proof of this claim can then be adapted to our setting in a straightforward way.
\end{rmk}

\section[The center of $\cO_\h(G)$]{The center of $\boldsymbol{\cO_\h(G)}$}
In this section we show the following:
\begin{thm}\label{thm:Poisson}
	The canonical inclusion of the $\kk[[\h]]$-linear extension of the Poisson center of~$\cO(G)$ into $\cO_\h(G)^\ca$ is an algebra morphism, and its image coincides with the center of $\cO_\h^\ca(G)$.
\end{thm}
Combined with the Hopf algebra isomorphism $\cO_\h^\ca(G)\cong \cO_\h(G)$ it implies the main result of this paper:
\begin{thm}\label{thm:main}
	The center of $\cO_\h(G)$ is isomorphic as an algebra to the Poisson center of $\cO(G)[[\h]]$.
\end{thm}
The rest of this section is devoted to the proof of Theorem~\ref{thm:Poisson}. We start with a few general facts about coalgebras in a braided monoidal category.
\begin{lem}\label{lem:1}
Let $C_1$, $C_2$ be two coalgebras in a braided tensor category $\cA$. Then the counit of~$C_1$ induces a coalgebra morphism from the braided tensor product $C_1\ot C_2$ to $C_2$.
\end{lem}
\begin{proof}
We have:
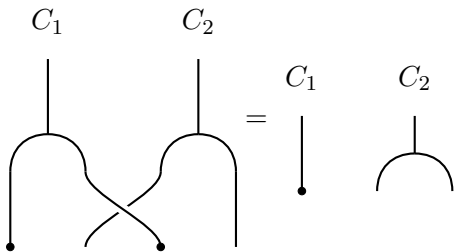
\begin{gather*}
\tik{\straight{0.5}{0.5} \straight{2.5}{0.5}
	\ap{0}{1}\ap{2}{1}
	\straight{0}{2}\negative{1}{2}\straight{3}{2}
	\node at (0.5,0) {$C_1$};
\node at (2.5,0) {$C_2$};
\node[point] at (0,-3) {};
\node[point] at (2,-3) {};
}=
\tik{
	\comult{1.5}{0} \straight{0.5}{0}
\node at (2,0.5) {$C_2$};
	\node at (0.5,0.5) {$C_1$};
\node[point] at (0.5,-1) {};
}\tag*{\qed}
\end{gather*}
\renewcommand{\qed}{}
\end{proof}

Therefore, the counit of $C_1$ induces a monoidal natural transformation from the functor $\hom_\cA(C_2,-)$ to the functor $\hom_\cA(C_1 \ot C_2,-)$, hence for every algebra $A$ in $\cA$ an algebra morphism
\begin{gather*}
\rho\colon \ \hom_\cA(C_2 ,A)\longrightarrow \hom_\cA(C_1 \ot C_2,A).
\end{gather*}
Indeed, for all $f,g\in \hom_\cA(C_2,A)$
\begin{gather*}
	\rho(f)\star \rho(g)=
	\tik{
		\comult{0}{0}\comult{2}{0}
	\straight{0}{1}\negative{1}{1}\straight{3}{1}
	\straight{1}{2}\straight{3}{2}
	\node[point] at (0,-2) {};
	\node[point] at (2,-2) {};
	\node[map] at (1,-2.5) {$f$};
	\node[map] at (3,-2.5) {$g$};
	\begin{scope}[xscale=2]
		\mult{0.5}{3}
	\end{scope}
	}=
	\tik{
		\comult{0}{-1}\straight{-1}{-1}\node[point] at (-1,0) {};
	\straight{0}{0} \node[rectangle,fill=white,draw] at (0,-0.5) {$f$};
	\straight{1}{0} \node[rectangle,fill=white,draw] at (1,-0.5) {$g$};
	\mult{0}{1}
}\ =\rho(f\star' g),
\end{gather*}
where $\star$ is the product in $\hom_\cA(C_1\ot C_2, A)$ induced from the monoidal structure on $\hom_\cA(C_1\ot C_2,-)$ and from the product of $A$, and where $\star'$ is def\/ined similarly for $\hom_\cA(C_2,A)$.
\begin{lem}\label{lem:2}
If $C_2$ is cocommutative and $A$ is commutative, then the image of $\rho$ is in the center of $\hom_\cA(C_1 \ot C_2,A)$.
\end{lem}
\begin{proof}
Let $f\in\hom_\cA(C_2,A)$ and $g \in\hom_\cA(C_1 \ot C_2,A)$. Then
\begin{gather*}
	\rho(f)\star g=\tik{
				\comult{0}{0}\comult{2}{0}
	\straight{0}{1}\negative{1}{1}\straight{3}{1}
	\straight{1}{2}
	\node[point] at (0,-2) {};
	\node[map] at (1,-2.5) {$f$};
	\triup{2}{2}{$g$}
	\begin{scope}[xscale=1.5]
		\mult{0.666666}{3}
	\end{scope}
	}
\end{gather*}
Using the commutativity of $A$ and the cocommutativity of $C_2$ this equals
\begin{gather*}
	\tik{
				\comult{0}{0}\comult{2}{0}
				\straight{0}{1}	\triup{1}{1}{$g$}\straight{3}{1}
	\node[map] at (3,-1.5) {$f$};
	\node[point] at (0,-2) {};
	\begin{scope}[xscale=1.5]
		\mult{1}{2}
	\end{scope}
			}
\end{gather*}
Then, by def\/inition of the counit one obtains
\begin{gather*}
	\tik{
				\comult{0}{0}\comult{2}{0}
	\straight{0}{1}\negative{1}{1}\straight{3}{1}
	\node[point] at (2,-2) {};
	\triup{0}{2}{$g$}\straight{3}{2}
	\node[map] at (3,-2.5) {$f$};
	\begin{scope}[xscale=2.5]
		\mult{0.2}{3}
	\end{scope}
}=g\star \rho(f).\tag*{\qed}
\end{gather*}
\renewcommand{\qed}{}
\end{proof}

We now prove Theorem~\ref{thm:Poisson}. By Proposition~\ref{prop:dressing} the Poisson center of $\cO(G)$ coincides with its subspace of $\mf g^*$-invariants. Therefore, by Frobenius reciprocity the $\kk[[\h]]$-extension of the Poisson center of $\cO(G)$ coincides with $\hom_\cA(M_-,\cO(G))$. Since $M_-$ is a coalgebra, the functor $\hom_\cA(M_-,-)$ is monoidal, and it is easily checked that this monoidal structure coincides with the canonical inclusion $V^{\mf g^*}\ot W^{\mf g^*}\subset (V\ot W)^{\mf g^*}$. In the same way the map
\begin{gather*}
	\hom_\cA(M_-,\cO(G))\rightarrow \hom_\cA(M_- \ot M_+,\cO(G))
\end{gather*}
induced by the counit of $M_+$ coincides with the inclusion $\cO(G)^{\mf g^*} \subset \cO(G)$. By Lemmas~\ref{lem:1} and~\ref{lem:2}, this is an algebra morphism whose image lies in the center of $\cO_\h^\ca(G)$.

	On the other hand, observe that any central element of $\cO_\h^\ca(G)$ belongs to the Poisson center of $\cO(G)[[\h]]$: if $f$ is central then for any function $g$ one has
	\begin{gather*}
		[f,g]=0=\h\{f,g\}+O\big(\h^2\big),
	\end{gather*}
	where $[f,g]$ is the commutator in the deformed algebra $\cO_\h^\ca(G)$. This completes the proof of Theorem~\ref{thm:Poisson}.

	\begin{rmk}\label{rmk:gamma}
		It is worth emphasizing that the map from the Poisson center of $\cO(G)$ to $\cO_\h(G)$ is not the standard inclusion. Indeed,, it relies on the isomorphism of Proposition~\ref{prop:monoidal}, which itself relies on the element $\beta$ chosen in Proposition~\ref{prop:dual}, which thus plays the role of the Duf\/lo element. Hence we expect $\beta$ to be closely related to the Duf\/lo function associated with $\Phi$ in the sense of~\cite{Alekseev2010}, also known as the Gamma function of $\Phi$.
	\end{rmk}

\section{The Duf\/lo isomorphism}
In this section we show that applying Theorem~\ref{thm:main} to the case where $G=\mf a^*$ for a f\/inite-dimensional Lie algebra $\mf a$, one recovers the Duf\/lo isomorphism. Let $\mf a_\h$ be the Lie algebra over~$\kk[[\h]]$ which is $\mf a[[\h]]$ as a module, and whose bracket is~$\h$ times the bracket of $\mf a$. Let $U(\mf a_\h)$ be its enveloping algebra. By the PBW theorem, there is a coalgebra isomorphism
\begin{gather*}
	S(\mf a)[[\h]]\cong U(\mf a_\h).
\end{gather*}
Pulling back the product of $U(\mf a_\h)$ through this isomorphism, one gets an algebra structure $\star_{\rm PBW}$ on $S(\mf a)[[\h]]$ compatible with the standard coproduct. It is easily seen that this is a quantization of the Poisson algebraic group $\mf a^*$. By~\cite[Lemma~3.2]{Enriquez2010a} this is the only functorial quantization of this Poisson group. This implies
\begin{prop}
	The Hopf algebra $\cO_\h(\mf a^*)$ is canonically, naturally isomorphic to
\begin{gather*}
(S(\mf a)[[\h]],\star_{\rm PBW},\Delta_0).
\end{gather*}
\end{prop}

By construction the product $\star_{\rm PBW}$ can be restricted to the subspace $S(\mf a)[\h]$ of polynomials in $\h$. This restriction can then be specialized at $\h=1$ and the resulting algebra is clearly the enveloping algebra of~$\mf a$. On the other hand, the Poisson center of $\cO(\mf a^*)=S(\mf a)$ coincide with the invariant under the adjoint action $S(\mf a)^{\mf a}$. This is well-known, and consistent with Proposition~\ref{prop:dressing} since in that case the dressing action of $\mf g^*=\mf a$ coincides with the adjoint action of $\mf a$ on $S(\mf a)$. One check that the isomorphism
\begin{gather*}
	\rho\colon \ \cO_\h(\mf a^*)^\ca\cong \cO_\h(\mf a^*)
\end{gather*}
coming from Theorem~\ref{thm:EK} also preserves the sub-space of polynomials in $\h$. Hence we get:
\begin{cor}[Duf\/lo isomorphism]
The isomorphism
\begin{gather*}
	S(\mf a)\cong U(\mf a)
\end{gather*}
given as the composition of $\rho$ with the PBW isomorphism, restricts to an algebra isomorphism
\begin{gather*}
	S(\mf a)^{\mf a}\cong U(\mf a)^{\mf a}.
\end{gather*}
\end{cor}

\begin{rmk}
It is known that the even part of a universal Duf\/lo isomorphism for Lie algebras is uniquely determined and coincide with the original Duf\/lo element introduced in~\cite{Duflo1977}. Hence, in the case of an even associator one recovers the original Duf\/lo isomorphism.
\end{rmk}

\pdfbookmark[1]{References}{ref}
\LastPageEnding

\end{document}